\documentclass{article}
\usepackage{amssymb}
\usepackage{amsthm}
\usepackage{ifthen,amsfonts,graphicx,latexsym,algorithm,algorithmic,url}
\usepackage{amsmath,color}
\usepackage{verbatim}
\usepackage[latin1]{inputenc}
\usepackage[english]{babel}

\newtheorem{theorem}{Theorem}[section]
\newtheorem{lemma}[theorem]{Lemma}
\newtheorem{proposition}[theorem]{Proposition}
\newtheorem{corollary}[theorem]{Corollary}

\newtheorem{defi}[theorem]{Definition\rm}
\theoremstyle{remark}
\newtheorem{remark}{Remark}

\newcommand{\ve}{\varepsilon}

\newcommand{\tori}{\mathbb T}
\newcommand{\jd}{\frac 12}
\newcommand{\deltat}{\partial_t}

\newcommand{\impl}{\Rightarrow}

\newcommand{\pical}{\mathcal{P}}
\newcommand{\lcal}{\mathcal{L}}

\newcommand{\R}{\mathbb{R}}

\newcommand{\ccr}{\mathbb{R}}
\newcommand{\ccrd}{\mathbb{R}^d}

\def\og{\leavevmode\raise.3ex\hbox{$\scriptscriptstyle\langle\!\langle$~}}
\def\fg{\leavevmode\raise.3ex\hbox{~$\!\scriptscriptstyle\,\rangle\!\rangle$}}

\title{Global-in-time regularity via duality\\ for congestion-penalized Mean Field Games}
\author{Adam Prosinski \thanks{University of Oxford, EPSRC CDT in Partial Differential Equations {\tt adam.prosinski@maths.ox.ac.uk} }, Filippo Santambrogio \thanks{Laboratoire de Math\'ematiques d'Orsay, Univ. Paris-Sud, CNRS, Universit\'e Paris-Saclay, 91405 Orsay Cedex, France {\tt filippo.santambrogio@math.u-psud.fr}}}

\begin{document}
\maketitle

\abstract{After a brief introduction to one of the most typical problems in Mean Field Games, the congestion case (where agents pay a cost depending on the density of the regions they visit), and to its variational structure, we consider the question of the regularity of the optimal solutions. A duality argument, used for the first time in a paper by Y. Brenier on incompressible fluid mechanics, and recently applied to MFG  with density constraints, allows to easily get some Sobolev regularity, locally in space and time. In the paper we prove that a careful analysis of the behaviour close to the final time allows to extend the same result including $t=T$. }

%

\section{Introduction}

The theory of Mean Field Games has been introduced some years ago by Lasry and Lions (in \cite{LL06cr1,LL06cr2,LL07mf}, see also \cite{GueLasLio} for a compendium of their results, and \cite{HCMieeeAC06} for a different point of view, by other authors, on this theory) 
to describe the evolution of a population, where each agent has to choose a trajectory in the state space, and has some preferences given in the form of a cost, but  this cost is affected by the other agents through a global mean field effect.

Mean Field Games (MFG for short) are differential games, with a continuum of players, usually all considered indistinguishable and such that any individual is negligible. We typically consider congestion games (i.e. agents try to avoid regions with high concentration), where we look for a Nash equilibrium, to be translated into a system of PDEs. In Section \ref{SecMod} we will see that this system is the coupling of a Hamilton-Jacobi equation on the value function of the optimal control problem of each agent, where the density of the population appears in the Hamiltonian, and of a continuity (transport) equation on the evolution of the population, where the velocity field depends on the gradient of the value function. 

MFG theory is now a very lively topic, and the literature is rapidly growing. Among the references for a general overview of the original developments of this theory, we recommend the videotapes of the 6-years course given by P.-L. Lions at Collège de France \cite{LLperso} and the lecture notes by P. Cardaliaguet \cite{Carnotes}, directly inspired by these courses.

In this paper we will only consider a particular class of MFG, the deterministic model with local coupling, which have a variational structure. We refer to \cite{BenCarSan} for a survey on this part of the theory of MFG.

After a brief introduction to the model, we will turn to its variational formulation, and recall the role of convex duality (Sections \ref{SecMod} and then \ref{SecPrel}). Later, in Section \ref{SecSpace}, we will present a recent method which uses duality to prove regularity results in a certain kind of variational problems. This method has, to the authors' knowledge, first been used in a paper by Y. Brenier on the Incompressible Euler equation in fluid mechanics (\cite{br}: later the results have been improved in \cite{af2}) and then adapted in \cite{Pierre} to the case of MFG with density constraints. It is however much more general, and \cite{LNRegDual} explains how to use it in order to recover, for instance, standard results in elliptic regularity.

In the model case (first studied in \cite{ButJimOud}), which amounts to the variational problem
$$\min\left\{\int_0^T\int_\Omega \jd( m|v|^2+m^2)dxdt+\int_\Omega \Psi(x) m(T,x)dx\;:\;\begin{cases}\partial_t m+\nabla\cdot(mv)=0,\\m(0,x)=\overline{m_0}(x)\end{cases}\right\},$$
we will show how we can provide local $H^1$ regularity for $m$. In particular, easy computations can give local regularity in time and global in space in the easier case where $\Omega=\mathbb T^d$ is the torus (more precisely, time derivatives are locally $L^2$ in $(0,T)$ and space derivatives are locally $L^2$ in $(0,T]$). For domains with boundary, it is possible to adapt the computations and obtain local regularity in space. This being said, the main content of the present paper is the global regularity in time, up to $t=T$ (Section \ref{SecTime}).

\section{Mean Field Games modelling: individual and collective optimization}\label{SecMod}
The goal of the MFG theory is to study the limit case $N \to \infty$ of a non-cooperative game of $N$ players who, over a fixed length of time $[0,T]$, move along curves $x: [0,T] \rightarrow \Omega$, trying to minimize the quantity
$$\int_0^T \left(\frac{|x_i'(t)|^2}{2}+g_i(x_1(t),\dots,x_N(t)))\right)dt + \Psi_i(x_i(T)).$$
The respective terms represent the costs in terms of kinetic energy, of the congestion induced by the other players, and the endpoint preferences. Indistinguishability of players means that the functions $\Psi_i$ are independent of $i$ and that the functions $g_i$ have a symmetric structure: 
$$g_i(x_1(t),\dots,x_N(t)))=g\left(x_i,\frac{1}{N}\sum_{j = 1}^N \delta_{x_j}\right).$$
With $g$ in this form it is quite easy to see what the limiting problem $N \rightarrow \infty$ should look like. All we do is to replace the discrete sum of Dirac delta measures by a general measure $m$ and give the instantaneous cost as $g(x,m)$. However, we note here that the relation of the two problems (finite and infinite number of players) is not within the scope of this paper. In particular, we do not prove any convergence results (we refer to \cite{CarDelLasLio} for some results in this direction, and to \cite{GueLasLio} for a general presentation) and proceed straight ahead with the analysis of the case of infinite number of players. 

Clearly, the cost function $g$ determines the nature of the problem. In this paper we consider a particular class of such functions, corresponding to a deterministic model with congestive local cost. That is, we want $g(x,m)$ to only depend on the behaviour of $m$ in a neighbourhood of $x$. In particular, if $m$ has some sort of a density $m(x)$ we want $g$ to be of the form $g(m(x))$, where $g: [0,\infty) \rightarrow [0, \infty]$ is increasing.

We note here that, in terms of regularity, the local congestion case is the most intriguing one. If the interaction is non-local (of the form $g((K*m)(x))$ for an interaction kernel $K$, so that the effective density perceived by the agents is of the form $\int K(x-y)m(y)dy$) then it automatically provides more compactness and regularity, results which are not available for local costs. We cite \cite{carda} as the first paper providing rigorous definitions and results for the local case. On the other hand, the non-local case is widely discussed in \cite{Carnotes}.

Now, for the strict formulation assume that we have a population of players, state of which is described by a time-dependent family of densities $m_t(\cdot)$. The players' actions are represented by curves $x(t)$ meant to solve
\begin{equation*}\label{cost}
\min\;\int_0^T \left(\frac{|x'(t)|^2}{2}+g(m_t(x(t)))\right)dt + \Psi(x(T)),
\end{equation*}
with fixed initial point $x(0)$. Here $g$ is increasing, which means that the players will try to avoid overcrowded regions, an assumption that is quite natural in a number of practical problems.

We begin the analysis by considering the function $h(t,x) := g(m_t(x))$ as a given of the problem. Thus, the players know the cost associated to each position in time and space, which brings us within the framework of standard optimal control theory. An essential tool here is the value function defined by
$$\varphi(t_0,x_0):=\min \left\{\int_{t_0}^T \left(\frac{|x'(t)|^2}{2}+h(t,x)\right)dt + \Psi(x(T)),\; x:[t_0,T]\to\Omega, x(t_0)=x_0\right\}.$$
It is known that $\varphi$ solves the following Hamilton-Jacobi equation
\begin{equation}\label{HJ}
\begin{cases}-\partial_t \varphi(t,x) + \frac 12|\nabla\varphi(t,x)|^2=h(t,x),\\ \varphi(T,x)=\Psi(x).\end{cases}
\end{equation}
Strictly speaking one needs some regularity assumptions on $h$ and $\Psi$ to establish this equation and even then it is usually only satisfied in the viscosity sense. However, our analysis here is purely formal and so we do not go into technical details. The important point is that if one can solve \eqref{HJ} then $\varphi$ may be retrieved from $h$ and $\Psi$. Once the value function is known it is easy to prove that optimal trajectories must solve
$$x'(t)=-\nabla\varphi(t,x(t)).$$
Finally, with the initial distribution $\overline{m_0}$ prescribed and players moving along solutions of the above ODE, a standard calculation (common in fluid mechanics for example) shows that $m$ and $\nabla\varphi$ must solve the so-called continuity equation:
\begin{equation}\label{CE}\partial_t m -\nabla\cdot (m \nabla \varphi)=0.\end{equation}

Going back to MFG, as usually in non-cooperative games, we look for a Nash equilibrium, that is, a configuration of strategies such that, having fixed the choices of others, no single player has an incentive to change his path. There are two possible equivalent approaches to finding such an equilibrium:

\begin{itemize}
\item First, we could consider the densities $m_t$ as initial data, solve the \eqref{HJ} equation with $h(t,x) = g(m_t(x))$ to find the optimal trajectories and then plug them into the continuity equation \eqref{CE}. The configuration is an equilibrium if and only if the solution of \eqref{CE} we get is precisely given by the $m_t$ we started with.
\item Alternatively we could start with players' trajectories as given and use them to determine the density evolution $m_t$ using \eqref{CE}, and insert the corresponding $h(t,x) = g(m_t(x))$ into \eqref{HJ}. Solving \eqref{HJ} for the optimal trajectories we may, once again, conclude that the configuration is an equilibrium if and only if the retrieved trajectories are exactly the ones we started with.  
\end{itemize}
In the end of the day both approaches boil down to finding a solution to the following \eqref{HJ} + \eqref{CE} coupled system:
\begin{equation}\label{MFGgrho}
\begin{cases}-\partial_t\varphi+\frac{|\nabla\varphi|^2}{2}=g(m),\\
		\partial_t m -\nabla\cdot (m \nabla \varphi)=0,\\
		\varphi(T,x)=\Psi(x),\quad m(0,x)=\overline{m_0}(x),\end{cases}
\end{equation}		
where there are terms in each equation depending on the solution of the other one. 

One of the possible strategies of tackling this system is to introduce a global minimization problem on the set of possible density evolutions. We set
  $$\mathcal B(m,v):=\int_0^T\int_\Omega\left( \frac 12m_t|v_t|^2 + G(m_t)\right)dxdt+\int_\Omega \Psi m_Tdx,$$
with $G$ being the anti-derivative of $g$, i.e. $G'=g$ on $\R^+$ with $G(0)=0$, so that in particular $G$ is convex, as its derivative is an increasing function. We seek to minimize $\mathcal{B}$ over all pairs $(m,v)$ satisfying $\partial_t m + \nabla(mv) = 0$, which is the Eulerian way of describing mass evolution. 

The above problem resembles the Benamou-Brenier dynamic formulation for optimal transport (see \cite{bb}). Observe however that in our case we introduce a congestion cost $G$, which is not present in the standard Benamou-Brenier set-up (but has been previously used to model the motion of a crowd in panic, see \cite{ButJimOud}) and we allow the end-time measure $m_T$ to vary. 

From the point of view of application
it is interesting to observe that the quantity $\mathcal{B}$ is not the total cost. Indeed, the terms $\int\int \frac 12m|v|^2$  and $\int \Psi m_T$ correspond to the total expenditure of kinetic energy and the global endpoint cost respectively, but the term $\int \int G(m_t)$ is not the same as the total congestion cost $\int \int m_t g(m_t)$. This shows that the equilibrium minimizes an overall energy (i.e. we work with a potential game), but not necessarily the total cost; this phenomenon is known as the presence of a price of anarchy. 

Another key idea is to introduce convexity to the problem (which has also been one of the main points of \cite{bb}). In what we have so far, the functional $\mathcal{B}$ is not convex due to the presence of the $m|v|^2$ term. The differential constraints are also problematic, as they involve the $mv$ product. However, a change of variables $(m,v) \rightarrow (m,w)$ with $w := mv$ makes the differential constraint linear. Furthermore, the function

$$\R\times\R^d\ni(m,w)\mapsto\begin{cases}\frac{|w|^2}{2m}&\mbox{ if }m>0,\\
									0&\mbox{ if }(m,w)=(0,0),\\
									+\infty&\mbox{ otherwise}\end{cases}$$
is convex (and equal to $\sup\{am+b\cdot w\;:\;a+\frac 12 |b|^2\leq 0\}$).

In this way we reduce our problem to minimizing a convex functional under linear differential constraints, which enables us to use a wide range of convex optimization techniques (see \cite{BenCar} for a numerical treatment), the first of which is the duality approach.

To identify the dual of the $\inf \mathcal{B}(m,w)$ problem we employ the min-max exchange procedure. First, we put the differential constraints in the weak form, i.e. we require that
  $$\int_0^T\int_\Omega(m\partial_tu +\nabla u\cdot m v)+\int_\Omega u_0m_0-\int_\Omega u_Tm_T=0$$
  for every function $u\in C^1([0,T]\times\Omega)$ (here $u_0$ and $u_T$ denote the functions $u(0,\cdot)$ and $u(T,\cdot)$, respectively). Note also that we do not impose any conditions on the values of $u$ on $\partial\Omega$, which is equivalent to completing \eqref{CE} with a no-flux boundary condition $v\cdot n=0$. We may now re-write our problem as
  $$\min_{m,v}\; \mathcal A(m,v)+\sup_u\int_0^T\int_\Omega(m\partial_tu +\nabla u\cdot m v)+\int_\Omega u_0\overline{m_0}-\int_\Omega u_Tm_T,$$
  since the sup in $u$ is $0$ if the constraint is satisfied and $+\infty$ if not. 
  
  Formally interchanging inf and sup and using the convex conjugate function $G^*(p)=\sup_m\; pm-G(m)$ we obtain a dual problem of the form
  $$\sup\left\{-\mathcal A(u,p):= \int_\Omega u_0\overline{m_0}-\int_0^T\int_\Omega G^*(p)\;:\;u\in C^1,\,u_T=\Psi,\,-\partial_tu+\frac12 |\nabla u|^2=p\right\}.$$
  Observe that we introduce an extra variable $p$ tied to $u$ through a PDE constraint even though we could have just used $u$. This is mainly to make the notation more natural, but also yields a certain symmetry with respect to the primal problem. The choice of the sign $- \mathcal{A}$ is arbitrary, but again, it makes the subsequent exposition clearer, as we will be interested in computing the sum $\mathcal A(u,p)+\mathcal B(m,v)$.
  
  We will see in Section \ref{SecPrel} that the set-up of the dual problems $\max-\mathcal A(u,p)$ and $\min\mathcal B(m,v)$ is indeed relevant to the MFG system \eqref{MFGgrho}, as optimizers immediately yield a formal solution to our PDEs. 
    
\section{Preliminary results on duality}\label{SecPrel}
First, we want to be precise on the duality result that we mentioned in Section \ref{SecMod}.

Denote by $\mathcal{D}_0$ the set of $C^1$ functions $u : [0,T] \times \Omega \rightarrow \ccr$ such that $u(T,x) =\Psi(x)$. We also denote by $\mathcal{D}$ the set of pairs $(u,p)$ with $u \in \mathcal{D}_0$, $p\in C^0([0,T] \times \Omega)$ and $p=-\partial_t u +\jd|\nabla u|^2$. Then, we define the functional
$$ \mathcal{A}(u,p) := \int_0^T \int_{\Omega} G^* (p) dx dt - \int_{\Omega} u_0 d\overline{m_0},$$
and we consider
$$ \inf\left\{ \mathcal{A}  (u,p)\in \mathcal{D}\right\}  \ldotp$$

The second problem concerns the set $\mathcal{H}$ defined as the set of pairs $(m,v)$ where $m \in L^1((0,T) \times \Omega; \ccr) $ and $v$ is a measurable vector field defined on $(0,T) \times \Omega$, such that $m(t,x) \geq 0$ a.e. on $(0,T) \times \Omega$, $\int_{\Omega} m(t,x) dx = 1$ for almost all $t \in (0,T)$, $mv\in L^1((0,T) \times \Omega; \ccr^d) $  and such that the pair $(m,v)$ satisfies, in the sense of distributions, the following continuity equation:
\begin{equation}\label{CEmv}
\begin{cases}
\deltat m +\nabla\cdot(mv) = 0, \\
mv\cdot n=0\,\mbox{ on }\partial\Omega,\\
m_0= \overline{m_0} \ldotp
\end{cases}
\end{equation}
In order to give a precise meaning to the boundary conditions above (both in time and space), we precise the meaning of \eqref{CEmv}: for every $\phi\in C^1_c([0,T)\times\overline\Omega)$ (i.e. we do not impose the support of $\phi$ to be far from $t=0$ or from $\partial\Omega$, but only from $t=T$) we require
$$\int_0^T\int_\Omega (m\partial_t\phi+mv\cdot\nabla\phi ) dxdt+\int_\Omega \phi(0,x)\overline{m_0}(x)dx=0.$$

For $(m,v) \in \mathcal{H}$ we define
\begin{equation}\label{eqDefB} 
\mathcal{B}(m,v) := \int_0^T \int_{\Omega}\left( \jd m(t,x)|v(t,x)|^2 +G(m(t,x))\right)dxdt + \int_{\Omega} \Psi(x)dm_T(x).
\end{equation}
Note that the first integral is well defined, with values in $[0,+\infty]$. To give meaning to the last term we observe that whenever 
$$\int_0^T \int_{\Omega}\jd m(t,x)|v(t,x)|^2 dxdt< \infty$$
then the theory of optimal transport (see Chapter 5 in \cite{OTAM}, for instance) allows us to identify $m$ with a continuous curve defined on $[0,T]$ and valued in $\pical(\Omega)$, endowed with the $W_2$ Wasserstein distance. Thus, $m_t$ is a well-defined measure for any $t \in [0,T]$, hence we have the right to refer to this measure when writing $\int_{\Omega}  \Psi(x)dm_T(x)$. 

Note that in this case we may also re-write \eqref{CEmv} as
$$\int_0^T\int_\Omega (m\partial_t\phi+mv\cdot\nabla\phi ) dxdt+\int_\Omega \phi(0,x)dm_0-\int_\Omega \phi(T,x)dm_T=0,$$
where $m_0=\overline{m_0}$ is prescribed, but $m_T$ is not.

Finally, we introduce our second optimization problem as 
$$ \inf_{(m,v) \in \mathcal{H}} \mathcal{B}(m,v) \ldotp$$
%

Let us state here the duality result we wish to use, and formulate the first assumptions we need for it to work:
\begin{description}
\item[(Hsuper) --] the function $G$ is superlinear, i.e. $\lim_{s\to\infty} G(s)/s=+\infty$.
\item[(Hstrict) --] the function $G$ is strictly convex.

\end{description}

\begin{theorem}\label{thmFirstDuality2}
If (Hsuper) holds, then
$$ \inf_{(u,p) \in \mathcal{D}} \mathcal{A}(u,p) = - \inf_{(m,v) \in \mathcal{H}} \mathcal{B}(m,v) \ldotp$$
and the infimum on the right-hand side is attained. The minimizer $(m,v) \in \mathcal{H}$ is unique if (Hstrict) also holds. 
\end{theorem}

The above results from the Fenchel-Rockafellar duality theorem. The reader can find its proof in \cite{carda}, Lemma 2.1, and should not be concerned by the growth conditions assumed in the beginning of  \cite{carda}, as they are not used in the duality proof (see also \cite{BenCarSan}). By abuse of language, we will call 'primal' the problem on $(m,v)$ and 'dual' the one on $(u,p)$, simply because we are more interested in $\min \mathcal{B}(m,v)$. Note that in general one does not expect the dual problem $\min \mathcal{A}$ to have a solution, at least not in $\mathcal{D}$. Yet, the problem can be relaxed so that it admits a solution in a more general class of functions (in the set of BV functions). We do not give the details of this reasoning here, however we intend to show that these optimization problems are intimately related to the solutions of the MFG system.

 The key here is the following:

\begin{lemma}\label{lemmaVariationalRegularity}
For any $(u,p) \in \mathcal{D}$ and $(m,v) \in \mathcal{H}$ we have 
$$\mathcal{A}(u,p) + \mathcal{B}(m,v) = \int_0^T \int_{\Omega}\left(G(m)+G^*(p)-mp\right) dx dt+ \frac{1}{2} \int_0^T \int_{\Omega} m |v + \nabla u|^2 dx dt \ldotp$$
\end{lemma}

\begin{proof}
We start with 
\begin{equation}\label{A+B}
\mathcal{A}(u,p) + \mathcal{B}(m,v) 
=\int_0^T \int_{\Omega}\left( \frac{1}{2} m|v|^2 + G(m) + G^*(p)\right) dxdt + \int_{\Omega} \Psi\, dm_T -  \int_{\Omega}  u_0\, d\overline{m_0}.
\end{equation}
Then we use
\begin{eqnarray*}
 \int_{\Omega} \Psi\, dm_T -  \int_{\Omega}  u(0)\, d\overline{m_0}&=& \int_0^T \int_{\Omega} \deltat (mu) dxdt= \int_0^T \int_{\Omega} \left(-u\nabla\cdot(mv) + m \deltat u \right)dx dt\\
 &=& \int_0^T \int_{\Omega} \left(\nabla u\cdot(mv) + m \left(\jd |\nabla u|^2-p\right)\right)dx dt.
 \end{eqnarray*}
Inserting this into \eqref{A+B} yields the desired result
$$\mathcal{A}(u,p) + \mathcal{B}(m,v) =\int_0^T \int_{\Omega}\left( \jd m|v + \nabla u|^2+G(m)+G^*(p)-mp\right)dx dt.$$
Observe that 
we used the fact that $u$ is $C^1$, when integrating by parts in  $\int_0^T \int_{\Omega} u\nabla\cdot(mv) $ (indeed, we just used the fact that $(m,v)$ satisfies \eqref{CE} in a weak sense). \end{proof}

Observe now that the above lemma coupled with Theorem \ref{thmFirstDuality2} immediately gives a solution to the MFG system as long as one assumes existence of a pair $(u,p)$ minimising the functional $\mathcal{A}$. Indeed, we then have $\mathcal{A}(u,p) + \mathcal{B}(m,v) = 0$, which yields
\begin{equation}
\begin{cases}
G(m)+G^*(p)-mp=0\;\impl\; p=g(m)\;\impl\;  - \deltat u + \jd |\nabla u|^2 =g(m)\quad \text{a.e.} \\
v = -\nabla u \quad m \text{-a.e.}
\end{cases}
\end{equation}
and the boundary conditions $ m_0=\overline{m_0}$ and $u(T,\cdot) = \Psi$ are automatically satisfied by all admissible $m$ and $u$. This shows that the functionals $\mathcal{A}$ and $\mathcal{B}$ are of importance from the MFG point of view. Note however that the above is only valid once we know that there exists a $C^1$ minimiser of $\mathcal{A}$, as we need to use \ref{lemmaVariationalRegularity}. If one extends the problem $\min \mathcal{A}$ to a larger class of admissible functions, then the notion of solution that we find should be weakened. We refer the reader to \cite{BenCarSan} for a survey about these notions, or to \cite{carda,CarGra} for the original papers. In particular we underline that the solutions to the MFG system would involve a BV function $u$, and the HJ equation would become an inequality in the region $\{m=0\}$.

In the following, we will use Lemma \ref{lemmaVariationalRegularity} to provide regularity properties on the minimizer $(m,v)$, in particular on $m$, of the primal problem.

Let us now precise the language that we will use to prove regularity. Consider the convex function $G$, and suppose that there exist two functions $J,J_*:\R\to\R$ and a positive constant $c>0$ such that for all $m,p\in\R$ we have
\begin{equation}\label{QP}
G(m)+G^*(p)\geq mp+c|J(m)-J_*(p)|^2.
\end{equation}

\begin{remark}
Of course, this is always satisfied by taking $J,J_*=0$, but we will be interested in less trivial cases, as we will later provide regularity results in terms of $J(m)$. Note for instance that we have the following interesting examples.
\begin{itemize}
\item If $G(m)=\jd |m|^2$, then we have $G^*=G$ and
$$\jd |m|^2+\jd |p|^2=mp + \jd |m-p|^2$$
hence we can take $J=J_*=id$.
\item If $G(m)=\frac 1q |m|^q$ for $q>1$, then $G^*(p)=\frac 1{q'} |p|^{q'}$, with $q'=q/(q-1)$; it is possible to prove
$$\frac 1q |m|^q+\frac 1{q'} |p|^{q'}\geq pq +\frac{1}{2\max\{q,q'\}} \left|m^{q/2}-p^{q'/2}\right|^2,$$
i.e. we can use $J(m)=m^{q/2}$ and $J_*(p)=p^{q'/2}$. 
\item If $G(m)=m\log m-m$ (defined as $+\infty$ for $m<0$), then $G^*(p)=e^p$ and we can prove the existence of a constant $c_0>0$ such that
$$m\log m-m+e^p\geq pm+c_0|\sqrt{m}-e^{p/2}|^2,$$
i.e.  we can use $J(m)=\sqrt{m}$ and $J_*(p)=e^{p/2}$. 
\item As a general fact, whenever we have $G''\geq c>0$, then we can use $J(m)=m$ and $J_*(p)=(G^*)'(p)=g^{-1}(p)$. 
\end{itemize}
Most of these inequalities are proved for instance in \cite{LNRegDual}. They are presented here in the case of scalar variables $m,p$ but can also be generalized to vector variables.
\end{remark}\smallskip

Before we proceed with the proof, let us briefly discuss the intuition behind what we will do. We wish to show that if $m$ is a minimiser of $\mathcal{B}$ then $J(m) \in H^1_{\text{loc}}((0,T] \times \Omega)$. The idea is that, should $\mathcal{A}$ admit a $C^1$ minimiser $\tilde{u}$ (more precisely, a pair $( u, p)$), then by the Duality Theorem \ref{thmFirstDuality2} we have $\mathcal{A}(u, p) + \mathcal{B}(m,v) = 0$. From our assumption and Lemma  \ref{lemmaVariationalRegularity}, we get $J(m) = J_*(p)$. If we managed to show that $\tilde{m}(t,x) := m(t + \eta, x + \delta)$ with corresponding field $\tilde{v}$ is close to minimising $\mathcal{B}$, in the sense
$$\mathcal{B}(\tilde{m}, \tilde{v}) \leq \mathcal{B}(m,v) + C(|\eta|^2 + |\delta|^2)$$
for small $\eta \in \ccr$, $\delta \in \ccrd$, then we would have
$$C(|\eta|^2 + |\delta|^2) \geq \mathcal{B}(\tilde{m}, \tilde{v}) + \mathcal{A}(u, p) \geq c || J(\tilde{m}) - J_*( p)||_{L^2}^2 \ldotp$$
However we already know that $J_*( p)=J(m)$, and so we get
$$C(|\eta|^2 + |\delta|^2) \geq c|| J(\tilde{m}) - J(m)||_{L^2}^2,$$
which would mean that $J(m)$ is $H^1$ as we have estimated the squared $L^2$ norm of the difference between $J(m)$ and its translation by the squared length of the translation vector. Of course this is just to give some intuition regarding what we do next. There are plenty of technical issues that need to be taken care of, for instance $\tilde{m}$ is not even well-defined (as we need the value of $m$ outside $[0,T]\times\Omega$), does not satisfy the initial condition $\tilde{m}_0 = \overline{m_0}$ and we do not know if $\mathcal{A}$ admits a minimiser. The aim of Sections \ref{SecSpace} and \ref{SecTime} is to deal with all these difficulties and arrive at the desired result.

We refer to \cite{BenCarSan} and \cite{Pierre} for some possible applications of the improved summability results that derive from the $H^1$ regularity that we can prove with these techniques, and in particular for a rigorous way to describe the Nash equilibrium at the trajectorial level (as it was first provided in \cite{af}).

\section{Space Regularity}\label{SecSpace}

In this section we will show how to obtain $\nabla (J(m))\in L^2_{loc}((0,T]\times\Omega)$, i.e. spatial regularity for $J(m)$. For simplicity, we will suppose that $\Omega=\tori^d$ is the flat torus, so that we do not have any boundary issue. Hence, when we say that the regularity result that we get is local, we mean local-in-time, as we will not prove it close to $t=0$. On the other hand, this is not surprising, as we did not assume any regularity on the initial datum $\overline{m_0}$. 

We do not develop here the regularity inside different domains $\Omega$, but we stress that this could be obtained (see \cite{LNRegDual}, for instance), far from the boundary $\partial\Omega$.

In order to prove regularity with respect to the $x$ variable and apply the ideas that we sketched at the end of the previous section, we would like to consider\\ 
$m^{\delta}(t,x) := m(t,x + \delta)$. Yet, such $m^{\delta}$ would not satisfy the initial condition of our continuity equation. Therefore we will use a cut-off function.

Fix any $\delta \in \ccrd$, a time instant $t_1 > 0$ and a smooth cut-off function $\zeta : [0,T] \rightarrow [0,1]$ with $\zeta \equiv 0$ on some neighbourhood of $0$ and $\zeta \equiv 1$ on $[t_1, T]$. Define 
\begin{equation}
\begin{cases}
m^{\delta} (t,x) := m(t, x + \zeta(t) \delta), \\
v^{\delta}(t,x) := v(t, x + \zeta(t) \delta) - \zeta'(t) \delta \ldotp
\end{cases}
\end{equation}
It is easy to check that the pair $(m^{\delta}, v^{\delta})$ satisfies the continuity equation together with the initial condition $m^{\delta}_0= \overline{m_0}$.
Therefore it is an admissible competitor in $\mathcal{B}$ with any choice of $\delta$. We may then consider the function
$$ M: \ccrd \rightarrow \ccr, \quad M(\delta) := \mathcal{B}(m^{\delta}, m^{\delta} v^{\delta}) \ldotp$$
The key point here is to show that $M$ is smooth (actually, we need $M\in C^{1,1}$). 

\begin{lemma}\label{MC11}
Suppose that $\Psi\in C^{1,1}$ and that $\Omega=\tori^d$. Then the function $\delta\mapsto M(\delta)$ defined above is $C^{1,1}$.
\end{lemma}
\begin{proof}
We have
$$\mathcal{B}(m^{\delta}, v^{\delta}) = \int_0^T\!\! \int_{\tori^d} \frac{1}{2} m^{\delta} |v^{\delta}|^2dx dt + \int_0^T\!\! \int_{\tori^d}G(m^{\delta}) dx dt+ \int_{\tori^d} \Psi(x) dm^{\delta}_T(x) \ldotp$$
Observe that the second integral does not depend on ${\delta}$ since it can be transformed, for each $t$, into the integral of $G(m)$, just by a translation change of variable. As for the last term, it can be written as
$$\int_{\tori^d} \Psi(x) dm^{\delta}_T(x) =\int_{\tori^d} \Psi(x-\delta) dm_T(x) $$
and it has at least the same regularity of $\Psi$. As we supposed $\Psi \in C^{1,1}$ (we will see at the end of Section \ref{SecTime} that this assumption can sometimes be weakened), this term is $ C^{1,1}$.
Now, we consider the first term
\begin{multline*}
 \int_0^T \int_{\tori^d} \frac{1}{2} m^{\delta}(t,x) |v^{\delta}(t,x)|^2 dx dt \\= \frac{1}{2} \int_0^T \int_{\tori^d} m(t, x +\zeta(t)\delta) ( |v(t, x+\zeta(t) \delta)|^2 - 2\zeta'(t) \delta \cdot v(t, x + \zeta(t)\delta)  + |\zeta'(t)|^2 |\delta|^2).
\end{multline*}
For a fixed $t$ we consider the change of variables in $x$ given by $x' = x+\zeta(t) \delta$. It is just a translation, and we work on a torus, thus we get
$$ \int_0^T \int_{\tori^d} \frac{1}{2} m^{\delta} |v^{\delta}|^2 dx dt =\frac{1}{2} \int_0^T \int_{\tori^d} m(t, x) ( |v(t, x)|^2 - 2\zeta'(t) \delta \cdot v(t, x)  + |\zeta'(t)|^2 |\delta|^2).$$
Since both $mv$ and $m$ are integrable, the function
$$ \delta \rightarrow \frac{1}{2} \int_0^T \int_{\tori^d} m(t,x) (- 2\zeta'(t) \delta \cdot v(t, x) + |\zeta'(t)|^2 |\delta|^2) dxdt$$
is smooth, and this proves that $M$ is $C^{1,1}$.
\end{proof}
We can now apply the previous lemma to get the estimate we need.
\begin{proposition}\label{propEstimateX} 
There exists a constant $C$, independent of $\delta$, such that for $|\delta| \leq 1$, we have
$$|M(\delta) - M(0)| = |\mathcal{B}(m^{\delta},v^\delta) - \mathcal{B}(m,v)| \leq C|\delta|^2 \ldotp$$
\end{proposition}
\begin{proof}
We just need to use Lemma \ref{MC11} and the optimality of $m$. This means that $M$ achieves its minimum at $\delta = 0$, therefore its first derivative must vanish at $0$ and we may conclude by a Taylor expansion: $M(\delta)=M(0)+\delta M'(t\delta)$ for $|t|<1$, where $|M'(t\delta)|=|M'(t\delta)-M'(0)|\leq Ct\delta$.
\end{proof}
With this result in mind, we can easily prove the following
\begin{theorem}\label{space-thm}
If $(m,v)$ is a solution to the primal problem $\min\mathcal B$, if $\Psi\in C^{1,1}$, if $\Omega=\tori^d$ and if $J$ is defined through \eqref{QP}, then $J(m)$ satisfies, for every $t_1>0$, 
$$||J(m(\cdot+\delta))-J(m)||_{L^2([t_1,T]\times\tori^d)}\leq C|\delta|$$
(where the constant $C$ depends on $t_1$ and on the data $\Psi$ and $G$), and hence
is of class $L^2_{loc}((0,T];H^1( \tori^d))$.
\end{theorem}
\begin{proof}
Let us take a minimizing sequence $(u_n,p_n)$ for the dual problem, i.e. $u_n\in C^1$, $p_n= -\deltat u_n + \frac{1}{2} |\nabla u_n|^2 $ and
$$\mathcal{A}(u_n,p_n) \leq \inf_{(u,p) \in \mathcal{D}} \mathcal{A}(u,p) + \frac{1}{n} \ldotp$$

We use $\tilde{m} = m^{\delta}$ and $\tilde v=v^\delta$ as in the previous discussion. Using first the triangle inequality and then Lemma \ref{lemmaVariationalRegularity} we have (where the $L^2$ norme denotes the norm in $L^2((0,T) \times \tori^d)$)
\begin{eqnarray*}
 c||J(m^{\delta}) - J(m)||_{L^2}^2 &\leq& 2( ||J(m^{\delta}) -J_*( p_n)||_{L^2}^2 + ||J(m) -J_*( p_n)||_{L^2}^2) \\
& \leq& 2 (\mathcal{A}(u_n,p_n) + \mathcal{B}(m^{\delta},  v^{\delta}) + \mathcal{A}(u_n,p_n) + \mathcal{B}(m,v)), 
\end{eqnarray*}
hence
$$||J(m^{\delta}) - J(m)||_{L^2}^2 \leq C(\mathcal{A}(u_n,p_n) + \mathcal{B}(m,v)) + C|\delta|^2  \leq \frac{C}{n} + C|\delta|^2   \ldotp
$$
Letting $n$ go to infinity and restricting the $L^2$ norm to $[t_1,T]\times \tori^d$, we get the claim. 
\end{proof}

\section{Time regularity}\label{SecTime}

In this section we would like to use the same idea of translations to get regularity in time. This can be done by defining $ m^{\varepsilon}_t:= m_{t-\varepsilon \zeta(t)}$ exactly as in the previous section, but we need to pay attention to the boundary $t=T$. Local regularity in time would be easy, but in order to prove a regularity result which arrives up to $t=T$, we need a more careful analysis.

First let us fix notation. We use some theory of Wasserstein spaces here, especially the notion of a metric derivative with respect to the Wasserstein distance. The reader is referred to \cite{AmbGigSav} or to Chapter 5 in \cite{OTAM}.

To make a very short summary of this theory, let us start from the classical Monge-Kantorovich problem.
Given two probability measures $\mu,\nu\in \pical(\Omega)$ we consider the set of transport plans
$$\Pi(\mu,\nu)=\{\gamma\in\pical(\R^d\times \R^d):\,(\pi_x)_{\#}\gamma=\mu,\,(\pi_y)_{\#}\gamma=\nu\},$$
i.e. the probability measures on the product space having $\mu$ and $\nu$ as marginal measures.

We consider the Kantorovich optimal transport problem for the cost $c(x,y)=|x-y|^2$ from $\mu$ to $\nu$, i.e.
$$ \min\left\{\int |x-y|^2\,d\gamma\;:\:\gamma\in\Pi(\mu,\nu)\right\}.$$
 
The value of this minimization problem with the quadratic cost may also be used to define the Wasserstein distance:
$$W_2(\mu,\nu):=\sqrt{ \min\left\{\int |x-y|^2\,d\gamma\;:\:\gamma\in\Pi(\mu,\nu)\right\}}.$$
On a compact $\Omega$, this quantity may be proven to be a distance over $\pical(\Omega)$, and it metrizes the weak-* convergence of probability measures. The space $\pical(\Omega)$ endowed with the distance $W_2$ is called Wasserstein space of order $2$ and is denoted in this paper by $\mathbb W_2(\Omega)$. 

We recall the definition of metric derivative in metric spaces, applied to the case of $\mathbb W_2(\Omega)$: for a curve $t\mapsto  m _t\in 
\mathbb W_2(\Omega)$, we define
$$|\dot{m}|(t):=\lim_{s\to 0}\frac{W_2( m _{t+s}, m _t)}{|s|},$$
whenever this limit exists. If the curve $t\mapsto  m _t$ is absolutely continuous for the $W_2$ distance, then this limit exists for a.e. $t$.
The important fact, coming from the Benamou-Brenier formula and explained for the first time in \cite{AmbGigSav}, is that the absolutely continuous curves in $\mathbb W_2(\Omega)$ are exactly those curves which admit the existence of a velocity field $v_t$ solving the continuity equation together with $ m $ and that the metric derivative $| \dot{m}|(t)$ can be computed as the minimal norm $||v_t||_{L^2( m _t)}$ among those vector fields. More precisely
\begin{proposition}
Suppose $( m ,v)$ satisfies the continuity equation $\partial_t m +\nabla\cdot( m  v)=0$ and $\int_0^T\int_\Omega  m |v|^2<\infty$. Then there exists a representative of $ m $ such that $t\mapsto  m _t\in \mathbb W_2(\Omega)$ is absolutely continuous and $| \dot{m}|(t)\leq ||v_t||_{L^2( m _t)}$ a.e. Conversely, if $t\mapsto m_t$ is an absolutely continuous map for the distance $W_2$, then for a.e. $t$ there exists a vector field $v_t\in L^2(m_t)$ such that $\partial_t m +\nabla\cdot( m  v)=0$ and $||v_t||_{L^2( m _t)}\leq | \dot{m}|(t)$.
\end{proposition}

This proposition allows us to get rid of the variable $v$ in the primal problem and recast it in the following form

$$\min\left\{\mathrm{B}(m) := \int_0^T \left(\jd |\dot{m}|(t)^2 + \mathcal G(m_t)\right) dt + \int \Psi dm_T\;:\; m_0=\overline{m_0}\right\},$$
where the functional $\mathcal G$ is defined through 
$$\mathcal G(m)=\begin{cases}\int G(m(x))dx&\mbox{ if }m\ll\lcal^d,\\
						+\infty &\mbox{ otherwise}.\end{cases}$$

Let us now fix $m$ to be a minimiser of $\mathrm{B}$. Take any $t_1 > 0$ and a smooth non-negative function $\zeta$ defined on $[0,T]$ and equal to $0$ on some neighbourhood of $0$.  For small $\varepsilon > 0$, the map $t \rightarrow t - \varepsilon \zeta(t)$ is a strictly monotone bijection of $[0,T]$ onto $[0, T- \varepsilon\zeta(T)]$. Define
$$ m^{\varepsilon}_t := m_{t-\varepsilon \zeta(t)} \ldotp$$

\begin{lemma}\label{computationzeta} 
We have
\begin{eqnarray*}
\mathrm B(m^{\varepsilon})
&=& \mathrm B(m) - \int_{T - \varepsilon\zeta(T)}^T\left( \jd |\dot{m}|(t)^2 + \mathcal G(m_t)\right) dt\\
&&+\int_0^{T-\varepsilon\zeta(T)}\left(\jd |\dot{m}|(s)^2 -\mathcal G(m_s)\right)\varepsilon \zeta'(s) ds\\
&& +  \int\Psi d(m_{T-\varepsilon\zeta(T)}-m_T)+ O(\varepsilon^2)\ldotp
\end{eqnarray*}
\end{lemma}
\begin{proof}
Observe that if we take $ s = t - \varepsilon \zeta(t),$
then we have
$$ s = t + O(\varepsilon),$$
$$ \varepsilon \zeta'(t) = \varepsilon \zeta'(s) + O(\varepsilon^2),$$
$$ \frac{1}{1 - \varepsilon \zeta'(t)} = 1 + \varepsilon \zeta'(s) + O(\varepsilon^2) \ldotp$$
Now let us calculate $\mathrm B(m^{\varepsilon})$. For notational simplicity, let us set $\ve_1=\ve\zeta(T)$. We have
$$ \mathrm B(m^{\varepsilon}) = \int_0^T \left(\jd |\dot{m}|(s)^2(1 - \varepsilon \zeta'(t))^2 + \mathcal G(m_s)\right) dt + \int\Psi dm_{T-\varepsilon_1} \ldotp$$
A change of variables yields
\begin{eqnarray*}
\mathrm B(m^{\varepsilon}) &=& \int_0^{T-\varepsilon_1}  \left(\jd |\dot{m}|(s)^2(1 - \varepsilon \zeta'(t)) + \mathcal G(m_s)(1 - \varepsilon \zeta'(t))^{-1} \right)ds +  \int\Psi dm_{T-\varepsilon_1} \\
&=& \int_0^{T-\varepsilon_1}  \jd |\dot{m}|(s)^2(1 - \varepsilon \zeta'(t)) + \mathcal G(m_s)(1 + \varepsilon \zeta'(t)) ds + \int\Psi dm_{T-\varepsilon_1}+ O(\varepsilon^2) \\
&=& \mathrm B(m) - \int_{T - \varepsilon_1}^T\left( \jd |\dot{m}|(t)^2 + \mathcal G(m_t)\right) dt\\
&&+\int_0^{T-\varepsilon_1}\left(\jd|\dot{m}|(s)^2 - \mathcal G(m_s)\right)\varepsilon \zeta'(s) ds +  \int\Psi d(m_{T-\varepsilon_1}-m_T)+ O(\varepsilon^2)\ldotp\qedhere
\end{eqnarray*}
\end{proof}

If we apply the computation above with $\zeta \in C_c^{\infty}((0,T))$, we can deduce the following
\begin{corollary}\label{diff=D}
If $m$ minimizes $\mathrm{B}$, then the quantity 
$$-\jd|\dot{m}|(t)^2 + \mathcal G(m_t)$$
is constant in time.
\end{corollary}
\begin{proof}
We use the computation of Lemma \ref{computationzeta} with $\zeta \in C_c^{\infty}((0,T))$ (in particular, $\zeta(T) = 0$), with 
$\varepsilon$ not necessarily positive, but still small. This yelds, for the corresponding $\tilde{m^{\varepsilon}}$,
$$\mathrm B(\tilde{m^{\varepsilon}}) = \mathrm B(m) + \int_0^{T} \left(\jd|\dot{m}|(s)^2 - \mathcal G(m_s)\right)\varepsilon \zeta '(s) ds + O(\varepsilon^2).$$
We can differentiate with respect to $\varepsilon$ and the optimality of $m$ gives
$$ \int_0^{T} (\jd |\dot{m}|(s)^2 - \mathcal G(m_s)) \zeta '(s) ds = 0\ldotp$$
This means that the difference $-\jd|\dot{m}|(s)^2 + \mathcal G(m_s)$ is constant.
\end{proof}
From now on, we will write
$$-\jd|\dot{m}|(t)^2 + \mathcal G(m_t)=D.$$

This allows us to rewrite the conclusions of Lemma \ref{computationzeta} as follows
\begin{eqnarray}\label{withD}
 \mathrm B(m^{\varepsilon}) &=& \mathrm B(m) - \int_{T - \varepsilon_1}^T \left(\jd |\dot{m}|(t)^2 + \mathcal G(m_t) \right)dt \\
 &&-\varepsilon D\zeta(T-\ve_1) +  \int\Psi d(m_{T-\varepsilon_1}-m_T)+ O(\varepsilon^2),\notag
 \end{eqnarray}
where we have used the fact that $\int_0^{T-\varepsilon_1} \zeta'(t) dt = \zeta(T-\ve_1) $.

We deduce now another consequence of the computations of Lemma \ref{computationzeta}.
\begin{lemma}\label{lastprep}
If $m$ is a minimiser of $\mathrm B$ then, with the notations above, we have
$$D  +\mathcal G(m_T) \leq  \jd \int_{\tori^d} |\nabla \Psi|^2 dm_T.$$
In particular,  $\mathcal G(m_T)<+\infty $ and $m_T$ it is absolutely continuous with respect to the Lebesgue measure.
\end{lemma}
\begin{proof}
We will use Lemma \ref{computationzeta} and formula \eqref{withD}
with a function $\zeta$ identically equal to $1$ on a neighbourhood of $T$. In this case we have
\begin{equation}\label{Bmve}
\mathrm{B}(m^{\varepsilon}) = \mathrm{B}(m) - \int_{T-\varepsilon}^T \left(\jd |\dot{m}|(t)^2 + \mathcal G(m_t)\right)dt - \varepsilon D + \int_{\tori^d}\Psi d(m_{T-\varepsilon}-m_T) + O(\varepsilon^2) \ldotp
\end{equation}
Let us consider the optimal transport plan $\gamma$ between $m_T$ and $m_{T-\varepsilon}$. It is a probability measure on  $\tori^d\times\tori^d$.
By a Taylor expansion on $\Psi$, we get
\begin{equation}\label{expPsi}
\int_{\tori^d}\Psi d(m_{T-\varepsilon}-m_T) \leq \int \nabla\Psi(x)\cdot (y-x) d\gamma+\frac C2\int |y-x|^2 d\gamma.
\end{equation}
%
Moreover, by the definition of the $2$-Wasserstein distance and the fact that $\gamma$ is optimal, we have
$$ \int |y-x|^2 d\gamma=  W_2^2(m_{T-\varepsilon},m_T) \leq  \left(\int_{T-\varepsilon}^T |\dot{m}|(t)dt\right)^2 \leq  \varepsilon  \int_{T-\varepsilon}^T |\dot{m}|(t)^2 dt, $$
where we have applied the Cauchy-Schwarz inequality in the last step. From this we have
\begin{equation}\label{gammamdot}
  \frac 1\ve \int |y-x|^2 d\gamma\leq  \int_{T-\varepsilon}^T |\dot{m}|(t)^2 dt.
\end{equation}
Inserting \eqref{expPsi} and \eqref{gammamdot} into \eqref{Bmve} and using the optimality of $m$, we get

$$\jd \frac 1\ve \int |y-x|^2 d\gamma +  \int_{T-\varepsilon}^T \mathcal G(m_t) dt + \ve D \leq  \int \nabla\Psi(x)\cdot (y-x) d\gamma+\frac C2\int |y-x|^2 d\gamma + O(\varepsilon^2) .$$
We apply Young inequality to get 
$$ \int \nabla\Psi(x)\cdot (y-x) d\gamma\leq \frac12 \frac{\ve}{1-C\ve}\int |\nabla\Psi(x)|^2 d\gamma+\frac12 \left(\frac 1\ve-C\right)\int |y-x|^2 d\gamma.$$
This gives
\begin{equation}\label{adiviser}
  \int_{T-\varepsilon}^T \mathcal G(m_t) dt + \ve D\leq   \frac12 \frac{\ve}{1-C\ve}\int |\nabla\Psi(x)|^2 d\gamma=  \frac12 \frac{\ve}{1-C\ve}\int |\nabla\Psi|^2 dm_T+ O(\varepsilon^2). 
\end{equation}

We now use the convexity of $\mathcal G$, which gives
\begin{equation}\label{mTve}
\mathcal G\left(  \frac{1}{\varepsilon} \int_{T-\varepsilon}^Tm_t dt\right)\leq\frac{1}{\varepsilon} \int_{T-\varepsilon}^T \mathcal G(m_t) dt.
  \end{equation}
For notational simplicity, we write $m_{T,\ve}$ for $ \frac{1}{\varepsilon} \int_{T-\varepsilon}^Tm_t dt$. Dividing \eqref{adiviser} by $\ve$ and inserting \eqref{mTve}, we have
$$\mathcal G(m_{T,\ve})+D\leq \frac{1}{2-2C\ve}\int |\nabla\Psi|^2 dm_T+ O(\varepsilon).$$ 
Since $m_t$ converge weakly in the sense of measures to $m_T$ with $t \rightarrow T$ so do the measures $m_{T,\ve}$ as $\ve\to 0$. Hence, we can take the $\liminf$ as $\ve\to 0$ of our last inequality and, using the semicontinuity of $\mathcal G$,
we get 
$$\mathcal G(m_T)+D\leq \frac{1}{2}\int |\nabla\Psi|^2 dm_T.\qedhere$$
\end{proof}

\begin{remark}\label{RemW2q'}
We note that this result allows us to weaken the assumptions on $\Psi$ necessary to prove that $\delta\mapsto \int\Psi(x-\delta)m(T,x)dx$ is $C^{1,1}=W^{2,\infty}$. Indeed, this function is a convolution, and if $m_T\in L^q$ then we only need $D^2\Psi \in L^{q'}$ in order to guarantee that their convolution is $L^\infty$. This means that, whenever $G(m)\geq c_0m^q-c_1$, i.e. it has $q$-growth, then $\Psi\in W^{2,q'}$ is enough for Theorem \ref{space-thm}.
\end{remark}

\begin{remark}
We observe that the term $\int_{\tori^d} |\nabla \Psi|^2 dm_T$ in the statement of the above lemma is the square of the slope, and also of the local lipschitz constant, in the $\mathbb W_2$ sense (see \cite{AmbGigSav,Gigli memoirs}), of the functional $m\mapsto \int_{\tori^d} \Psi dm$, computed at $m=m_T$. We recall the definitions of the slope for a functional $F$ defined on a metric space
$$|\nabla^- F|(x):=\limsup_{y\to x}\frac{[F(x)-F(y)]_+}{d(x,y)}$$
and of the local Lipschitz constant
$$\mathrm{lip}F(x):=\limsup_{y\to x}\frac{|F(x)-F(y)|}{d(x,y)}.$$
Indeed, a similar result to the one we presented could be obtained for optimal curves in arbitrary metric space: we can prove, under suitable assumptions, that the minimizers of $\int_0^T( \jd |\dot\omega|^2(t)+G(\omega(t)))dt+F(\omega(T))$ satisfy $ -\jd |\dot\omega|^2(t)+G(\omega(t))=D$ and $D+G(\omega(T))\leq \jd(\mathrm{lip}F(\omega(T)))^2$. Here we preferred to exploit the explicit structure of our functionals to avoid using and checking general definitions.
\end{remark}
\bigskip

We now proceed to the key part of this section. We want to consider $m^{\varepsilon}_t:= m_{t + \varepsilon \zeta(t)}$ for $\varepsilon > 0$ and with $\zeta(T) = 1$ (pay attention to the change in the sign!). However, to do so we need to extend $m$ onto $[T, T+\varepsilon]$. To this end we consider, for $\delta \in [0,\varepsilon]$ and any Lipschitz vector field $v$, the function 
$$ R_{\delta} (x) := x - \delta v(x)$$ 
and we let
$$ m_{T + \delta} := (R_{\delta})_\#m_T \ldotp$$
Observe that once we fix $v$ then for small $\varepsilon$ the function $x - \delta v(x)$ is a Lipschitz homeomorphism for all $\delta \in [0,\varepsilon]$. Moreover
$$ m^{\varepsilon}_{T - \varepsilon + \delta}(R_\delta(x)) =m_{T + \delta}(R_\delta(x))=  \frac{m_T(x)}{\text{det} (\text{Id} - \delta Dv(x))}\quad a.e. \ldotp$$
The term in the denominator above is the Jacobian of the change of variables $ x \rightarrow R_{\delta}(x)$. Furthermore
$$ \text{det} (\text{Id} - \delta Dv) = 1 - \delta(\nabla\cdot v) + O(\delta^2) \ldotp$$

In order to prove our next proposition we need to introduce another assumption on $G$, namely
\begin{description}
\item[(Hpol)--] There exist two numbers $C,a_0>0$ such that, for all $m>0$ and $a<a_0$, the function $G$ satisfies the following inequality
$$G((1+a)m)\leq (1+Ca)G(m)+C.$$
\end{description}
Note that this assumption is satisfied by all convex functions $G$ satisfying bounds of the form $c_1m^q-c_2\leq G(m)\leq c_3m^q+c_4$ for $q\geq 1$, and also by $G(m)=m\log m$.
\begin{proposition}\label{propTranslateM}
Suppose that $\Psi \in C^{1,1}$ and that $G$ satisfies (Hpol). Then, if we define $m^\ve$ as above, and choose $v=\nabla \Psi$, there exists a finite constant $C$ independent of $\varepsilon$ such that for $\varepsilon$ sufficiently small we have
$$|\mathrm{B}(m) - \mathrm{B}(m^\ve)| \leq C|\varepsilon|^2 \ldotp$$
\end{proposition}
\begin{proof}
Calculating as before we get
$$\mathrm{B}(m^{\varepsilon}) = \mathrm{B}(m) + \varepsilon D +  \int_{T}^{T+\varepsilon}\jd |\dot{m}|(t)^2dt +\int_{T}^{T+\varepsilon} \mathcal G(m_t) dt +\int \Psi d(m_{T+\varepsilon}-m_T) + O(\varepsilon^2) \ldotp$$
We now need to estimate the different terms of this sum.

Let us start from
$\jd \int_{T}^{T+\varepsilon} |\dot{m}|(t)^2 dt$. Here we know that the curve $m$ is obtained by moving particles with constant speed $v$ (in Lagrangian coordinates), even though we do not know if this $v$ is the optimal vector field in the continuity equation. Hence, we have 
$$ |\dot{m}|(t)\leq ||v||_{L^2(m_T)},$$ 
which gives
$$\jd \int_{T}^{T+\varepsilon} |\dot{m}|(t)^2 dt \leq \jd \varepsilon \int_{\tori^d} |v(x)|^2 dm_T \ldotp$$

For the second integral note that, as stated before, at any fixed time $T + \delta$ we have, by the definition of a push-forward measure,
$$\int_{\tori^d} G(m(T+\delta)) dx = \int_{\tori^d} \frac{G(m(T+\delta))}{m(T+\delta)} dm(T+\delta) =  \int_{\tori^d} \frac{G(m(T+\delta))}{m(T+\delta)} (R_{\delta}(x)) dm_{T}.$$
Here we note that $s\mapsto G(s)/s$ is an increasing function (since $G$ is convex and $G(0)=0$) and 
we use Assumption (Hpol) to get
$$\int_{\tori^d} G(m_{T+\delta}) dx =  \int_{\tori^d} G(m_T) dx + O(\delta) \ldotp$$
Integrating this on $[T, T+\varepsilon]$ yields
$$ \int_{T}^{T+\varepsilon} G(m_t) dt = \varepsilon \int_{\tori^d} G(m_T) dx + O(\varepsilon^2) \ldotp$$

We now look at the last term, for which we have
$$\int \!\Psi d(m_{T+\varepsilon}-m_T)=\int_{\tori^d} \!(\Psi(x - \varepsilon v(x)) - \Psi(x)) dm_T = \int_{\tori^d} \!\nabla \Psi(x)\cdot (-\varepsilon v(x)) dm_T + O(\varepsilon^2),$$
where we used a second-order Taylor expansion for $\Psi$.

Putting this all together gives
$$\mathrm{B}(m^{\varepsilon}) \leq \mathrm{B}(m) + \varepsilon D + \jd \varepsilon \int_{\tori^d} |v|^2 dm_T +  \varepsilon  \mathcal G(m_T) -\ve \int_{\tori^d} \nabla \Psi(x)\cdot v(x) dm_T + O(\varepsilon^2) \ldotp$$
Choosing $v(x) := \nabla \Psi(x)$ (which is an admissible choice because of the regularity assumptions on $\Psi$) this gives
$$ \mathrm{B}(m^{\varepsilon}) \leq \mathrm{B}(m) + \varepsilon D + \jd \varepsilon \int_{\tori^d} |\nabla \Psi|^2 dm_T + \varepsilon \mathcal G(m_T)  - \varepsilon \int_{\tori^d} |\nabla \Psi|^2 dm_T + O(\varepsilon^2) = $$
\begin{equation}\label{eqEstimateFinish} = \mathrm{B}(m) + \varepsilon\left( D - \jd \int_{\tori^d} |\nabla \Psi|^2 dm_T + \mathcal G(m_T)\right) + O(\varepsilon^2) \ldotp
\end{equation}
The claim is proven using Lemma \ref{lastprep} and the 
 fact that necessarily $\mathrm{B}(m) \leq \mathrm{B}(m^{\varepsilon})$.
\end{proof}

This allows us to conclude the following:
\begin{theorem}
If $(m,v)$ is a solution to the primal problem $\min\mathcal B$, if $\Psi\in C^{1,1}$ and $G$ satisfies the assumptions (Hsuper) and (Hpol), if $\Omega=\tori^d$ and if $J$ is defined through \eqref{QP}, then $J(m) \in H^1_{loc}((0,T]\times \tori^d)$.
\end{theorem}
\begin{proof}
The proof follows from the strategy described in Sections 3 and 4, using Proposition \ref{propTranslateM}. This proves the estimates on the time-derivative. The space-derivatives are already estimated thanks to Theorem \ref{space-thm}.
\end{proof}

\begin{remark}
Among the consequences of these $H^1$ regularity results, we insist on the summability improvement. For instance, whenever the penalization is $G(m)=m^q$, the fact that the optimal $m$ satisfies $m\in L^q([0,T]\times \tori^d)$ is straightforward. However, with $J(m)=m^{q/2}$ and using the injection $H^1\subset L^{2^*}$, we also get $m\in L^{q2^*/2}$. This is also important in the case $G(m)=m\log m-m$, where one obtains $m\in L^{2^*/2}$: this is especially useful when one needs to exit the space $L^1$, where many properties lack (see, for instance, \cite{BenCarSan} or \cite{Pierre}, to see applications where integrability properties of the maximal function are required). Note, however, that the exponent $2^*$ should be computed here w.r.t. to the dimension $d+1$ of $[0,T]\times \tori^d$; on the other hand, it is possible to use Theorem \ref{space-thm} and obtain $J(m)\in L^2_{loc}((0,T]; L^{2^*}( \tori^d))$, where here $2^*$ is computed only using the space variable (but the integrability in time is not improved).
\end{remark}

 \paragraph{Acknowledgments}
 The first author worked on this topic during his master studies at Ecole Polytechnique, funded by a Scholarship by Fondation Mathématique Jacques Hadamard, whose support is acknowledged.
 The second author acknowledges the support of the ANR project ISOTACE (ANR-12-MONU-0013) and of the iCODE project ``Strategic Crowds'' funded by IDEX Paris-Saclay, and warmly thanks the organizers of International Conference on Stochastic Analysis and Applications, Hammamet, October 2015, for the opportunity to present the results and publish them.

\end{document}